\documentclass[11pt]{article}
\UseRawInputEncoding
\usepackage{amssymb}
\usepackage{geometry}
\usepackage{graphicx}
\usepackage{listings}
\usepackage{color}

\usepackage{fancyhdr}				
\usepackage{lastpage}				

\usepackage[colorlinks,linkcolor=blue]{hyperref}									
\usepackage{indentfirst}			
\usepackage{latexsym,bm}        	
\usepackage{amsmath,amssymb}    	
\usepackage{graphicx}
\begin{document}
\setlength{\textwidth}{6.0in}
\setlength{\textheight}{9.0in}

\setlength{\oddsidemargin}{0.0cm}
\setlength{\evensidemargin}{0.0cm}
\setlength{\topmargin}{0.0cm}

 \newtheorem{theorem}{Theorem}[section]
 \newtheorem{lemma}[theorem]{Lemma}
 \newtheorem{proposition}[theorem]{Proposition}
 \newtheorem{question}[theorem]{Question}
 \newtheorem{corollary}[theorem]{Corollary}
 \newtheorem{definition}[theorem]{Definition}
 \newenvironment{proof}{\begin{trivlist} \item[]{\em Proof.}}{\end{trivlist}}
 \newcount\refno
 \refno=0


 \def\Ab{{\mathbf A}}
 \def\ab{{\mathbf a}}
 \def\bb{{\mathbf b}}
 \def\eb{{\mathbf e}}
 \def\kb{{\mathbf k}}
 \def\bo{{\mathbf o}}
 \def\pb{{\mathbf p}}
 \def\rb{{\mathbf r}}
 \def\sb{{\mathbf s}}
 \def\tb{{\mathbf t}}
 \def\ub{{\mathbf u}}
 \def\vb{{\mathbf v}}
 \def\xb{{\mathbf x}}
 \def\yb{{\mathbf y}}
 \def\zb{{\mathbf z}}
 \def\Kb{{\mathbf K}}
 \def\Pb{{\mathbf P}}
 \def\CA{{\mathcal A}}
 \def\CB{{\mathcal B}}
 \def\CD{{\mathcal D}}
 \def\CF{{\mathcal F}}
 \def\CI{{\mathcal I}}
 \def\CL{{\mathcal L}}
 \def\CH{{\mathcal H}}
 \def\CM{{\mathcal M}}
 \def\CO{{\mathcal O}}
 \def\CP{{\mathcal P}}
 \def\CS{{\mathcal S}}
 \def\CV{{\mathcal V}}
\def\CC{{\mathbb C}}
\def\DD{{\mathbb D}}
\def\AA{{\mathbb A}}
\def\BB{{\mathbb B}}
 \def\PP{{\mathbb P}}
 \def\RR{{\mathbb R}}
 \def\NN{{\mathbb N}}
 \def\SB{{\mathbb S}}
 \def\TT{{\mathbb T}}
 \def\ZZ{{\mathbb Z}}
\def\SD{{\mathscr D}}
 \def\SF{{\mathscr F}}
 \def\SH{{\mathscr H}}
  \def\SM{{\mathscr M}}
 \def\SS{{\mathscr S}}
        \def\proj{\operatorname{proj}}
        \def\loc{\operatorname{loc}}
        \def\const{\operatorname{const}}


 \def\thefootnote{}
 \title{\bf Kakeya  inequalities by maximal functions in  Hardy spaces
 \thanks{\indent\,\, Department of Mathematics, Capital Normal University, Beijing 100048, China,
 E-mail: huzhuoran010@163.com.}}

\author{{ZhuoRan Hu}}

 \maketitle \setcounter{page}{1} \pagestyle{myheadings}
 \markboth{Hu}{Kakeya  inequalities by  maximal functions in  Hardy spaces  }

 \begin{abstract}
In this paper, we will introduce and study several types of Kakeya inequalities  by the maximal functions in  Hardy spaces in $\RR^n$,\,$(n\geq2)$, and we could obtain several  inequalities associated with the Kakeya inequalities. We will show that $\big\|M^t_{\delta S_{\alpha,\beta}} f \big\|_p\lesssim_{p,n,\varphi,\varepsilon}\left(\frac{1}{\delta}\right)^{5\left(\frac{n}{r}+2\right)\varepsilon}\big\|f\big\|_p$, when $f(x)\in L^p(\RR^n)$  and $supp\,\hat{f}(\xi)\subseteq B(0,1)$.

 \vskip .2in
 \noindent
 {\bf 2000 MS Classification:} 42B20, 42B25, 42A38.
 \end{abstract}

\setcounter{page}{1}

\section{Introduction }

In 1917, Kakeya\,\cite{SK}  proposed a problem to determine the minimal area needed to continuously rotate a unit line segment in the plane by 180 degrees. In 1928, Besicovitch\,\cite{BS} proved the measure of such sets could be arbitrary small. Such sets are called Besicovitch Sets or Kakeya Sets. The Kakeya conjectures states that the Hausdorff dimension of any Besicovitch Sets in $\RR^n$ is n. The case for $n\geq3$ is still an open problem. The so-called maximal Kakeya conjecture (or maximal Nikodym conjecture) is actually a stronger one that involves the following Kakeya maximal function (or Nikodym maximal function):
\begin{eqnarray}\label{37}
f_\delta^*(\xi)=\sup_{a\in\RR^n}\frac{1}{|T_\xi^{\delta}(a)|}\int_{T_\xi^{\delta}(a)}|f(y)|dy,
\end{eqnarray}
where $T_\xi^{\delta}(a)$ is a $1\times\delta$ tube centered at $a\in \RR^n$ with the direction $\xi\in S^{n-1}$.
\begin{eqnarray}\label{38}
f_\delta^{**}(x)=\sup_{x\in T}\frac{1}{|T|}\int_{T}|f(y)|dy,
\end{eqnarray}
where the supremum is taken over all $1\times\delta$ tubes $T$ that contain $x\in\RR^n$. Formula\,(\ref{37}) is Kakeya maximal function and
Formula\,(\ref{38}) is Nikodym maximal function.
When $n=2$, in \cite{Co}, Cordoba proved that for any $\varepsilon>0$ $$\|f_\delta^*\|_{L^{2}(S^1)}\lesssim_{\varepsilon} \delta^{-\varepsilon}\|f\|_{L^{2}(\RR^{2})}.$$
The Kakeya maximal function conjecture is formulated by Bourgain\,\cite{Du3} that
\begin{eqnarray}\label{39}
\|f_\delta^*\|_{L^{p}(S^{n-1})}\lesssim_{\varepsilon} \delta^{-\varepsilon}\|f\|_{L^{p}(\RR^{n})}
\end{eqnarray}
holds for $p\geq n$ and $n\in\NN$, and
\begin{eqnarray}\label{40}
\|f_\delta^*\|_{L^{q}(S^{n-1})}\lesssim_{\varepsilon} \delta^{-\frac{n}{p}+1-\varepsilon}\|f\|_{L^{p}(\RR^{n})}
\end{eqnarray}
holds for $1<p\leq n$, $q=(n-1)p'$ and $n\in\NN$. In 1983, Drury proved Formula\,(\ref{40}) for $p=(d+1)/2$, $q=n+1$ in \cite{SD}. In 1991, Bourgain in
\cite{Du3} improved this result for each $n\geq3$ to some $p(d)\in((d+1)/2, (d+2)/2)$.
By the interpolation theory,
(see \cite{x,y,z} and reference therein)
\begin{eqnarray}\label{44}
\|f_\delta^*\|_{L^{p}(\RR^{n})}\lesssim_{\varepsilon} \delta^{-\frac{n-1}{p}-\varepsilon}\|f\|_{L^{p}(\RR^{n})}
\end{eqnarray}
holds for $p\geq n$ and $n\in\NN$.

 \textbf{Main result}:  Inspired by the Formulas\,(\ref{37},\,\ref{38},\,\ref{39},\,\ref{40},\,\ref{44}), we will consider maximal functions like  $M_{\delta S_{\alpha,\beta}} f(x)$ and $M^t_{\delta S_{\alpha,\beta}} f(x)$ in this paper. Notice that the classical case is $\delta=1$, then $M_{1 S_{\alpha,\beta}} f(x)$ and $M^t_{1 S_{\alpha,\beta}} f(x)$ are classical maximal functions in Hardy spaces. And
 $$\big\|M_{1 S_{\alpha,\beta}} f \big\|_p\leq C \big\|(f\ast\varphi)_\bigtriangledown\big\|_p,\ \ \big\|M_{1 S_{\alpha,\beta}} f \big\|_p\leq C\|f\|_p.$$
for some constant $C>0$.

 We will obtain several  inequalities in Proposition\,\ref{25},\, Theorem\,\ref{35} and Theorem\,\ref{50}. In Proposition\,\ref{25}, though  the coefficient  in Formula\,(\ref{42})  is not better than the  factor $\delta^{-\frac{n-1}{p}-\varepsilon}$ in Formula\,(\ref{44})£¬ but we use a  way   different to \cite{Du3}, \cite{Co} and \cite{SD}.  And we could obtain  Formula\,(\ref{41})   which is different to the classical case $\delta=1$. In Theorem\,\ref{35}, the coefficient is  the same as the factor  $\delta^{-\varepsilon}$ in Formula\,(\ref{39}) when $supp\,\hat{f}(\xi)\subseteq B(0,1)$. In  Theorem\,\ref{50}, the coefficient is  independent on $\delta$ when  $supp\,\hat{f}(\xi)\subseteq \left(B(0,t^{-1}\delta^{-(1+4\varepsilon)})\right)^c$.\\

\textbf{Notation}: As usual, we  use $n$ to denote the dimension of  $\RR^n$. $supp\,f(x)$  is the support set of $f(x)$. If $x\in\RR^n$: $x=(x_1,x_2,\cdots,x_n)$, $|x|_e$ denotes  $|x|_e=\sqrt{x_1^2+x_2^2+\cdots+x_n^2}.$
For $\alpha\in\NN^n$:$\alpha=(\alpha_1,\alpha_2,\cdots,\alpha_n)$,  $|\alpha|$  denotes  $|\alpha|=\alpha_1+\alpha_2+\cdots+\alpha_n.$ We use $\|.\|_{p}$ to denote $\|.\|_{L^p(\RR^n)}$, and $O(\RR^n)$ to denote the $n\times n$ unit orthogonal matrix in $\RR^n$:
$O(\RR^n)=\{A:  A^{T}A=1. \,A^{T}\,\hbox{is\,the\,transposed\,matrix\,of}\,A\}.$
 $S(\RR^n)$ designates the space of $C^{\infty}$ functions on $\RR^n$ rapidly decreasing together with their derivatives. $S_{\alpha,\beta}(\RR^n)$ denotes :
$S_{\alpha,\beta}(\RR^n)=\{\phi\in S(\RR^n):  \|\phi\|_{\alpha', \beta'}\leq1,\,\, \forall\alpha',\beta'\in\NN^n, |\alpha'|\leq |\alpha|,|\beta'|\leq |\beta|\}$, and $\varepsilon$   a positive fixed number (may be very small):  $\varepsilon>0$.

If $X$ and $Y$ are two quantities,  $X\lesssim Y$ or $Y\gtrsim X$ denotes that $X\leq CY$
for some absolute constant $C>0$.  More generally, given some
parameters $a_1,\cdots,a_k$, we use $X\lesssim_{a_1,\cdots,a_k} Y$ or $Y\gtrsim_{a_1,\cdots,a_k} X$ to denote the statement that $X\leq C_{a_1,\cdots,a_k}Y$ for some constant $C_{a_1,\cdots,a_k}$ which can depend on the parameter $a_1,\cdots,a_k$. We use
$X\sim Y$ to denote the statement $X\lesssim Y\lesssim X$, and similarly $X\sim_{a_1,\cdots,a_k} Y$ denotes $X\lesssim_{a_1,\cdots,a_k} Y\lesssim_{a_1,\cdots,a_k} X$.

\section{Preliminaries}
For $t,\xi\ \in\RR^n$, $f\in S(\RR^n)$, the Fourier transform of $f$ is given by
$$\hat{f}(\xi)=\mathfrak{F}f(\xi)=\int_{\RR^n} f(t)e^{-2\pi i <\xi, t>}dt,$$
thus
$f(x)=(\hat{f})^{\vee}(x),$
where $<\xi, t>=\sum_{k=1}^{n} \xi_k t_k$ and $\vee$ is the Inversion of Fourier transform. For $g\in S(\RR^n)$, $g_I(x)$  designates
$$g_I(x)=\left(\frac{1}{\delta}\right)^{n-1}g\left(\frac{x_1}{\delta}, \frac{x_2}{\delta}, \cdots,\frac{x_{n-1}}{\delta}, x_{n} \right).$$
Let $u=(u_1,u_2,\cdots,u_n)=xA^{-1}=(x_1,x_2,\cdots,x_n)A^{-1}$ where $A$ is a variable (not fixed) matrix with $ A\in O(\RR^n)$,  then $g_{AI}(x)$ is given by
\begin{eqnarray*}
g_{AI}(x)=g_I\left(xA^{-1}\right)
=g_I\left(u\right)
=\left(\frac{1}{\delta}\right)^{n-1}g\left(\frac{u_1}{\delta}, \frac{u_2}{\delta}, \cdots,\frac{u_{n-1}}{\delta}, u_{n} \right).
\end{eqnarray*}
If $A$ is a variable (not fixed) matrix with $ A\in O(\RR^n)$, let
\begin{eqnarray*}
\zeta=\left(\begin{array}{c}
\zeta_1 \\
\zeta_2 \\
\vdots \\
\zeta_n
\end{array} \right)=A\xi=A\left(\begin{array}{c}
\xi_1 \\
\xi_2 \\
\vdots \\
\xi_n
\end{array} \right),
\end{eqnarray*}
thus
$$\mathfrak{F}(g_I)(\xi)=\mathfrak{F}(g)(\delta\xi_{1},\delta\xi_{2},\ldots,\delta\xi_{n-1}, \xi_n),$$
and
$$\mathfrak{F}(g_{AI})(\xi)=\mathfrak{F}(g_{I})(A\xi)=\mathfrak{F}(g_{I})(\zeta)=\mathfrak{F}(g)(\delta\zeta_{1},\delta\zeta_{2},\ldots,\delta\zeta_{n-1}, \zeta_n).$$
In this paper, let $\varphi \in S(\RR^n)$ always to be  a fixed radial function satisfying the following:
\begin{align*}\left\{\begin{array}{lll}
\widehat{\varphi}(\xi)=1,\ \hbox{for}\  |\xi|_e\leq1,\\\\
\widehat{\varphi}(\xi)=0,\ \hbox{for}\  |\xi|_e\geq2,\\\\
\widehat{\varphi}(A\xi)=\widehat{\varphi}(\xi)\ \hbox{for}\ A\in O(\RR^n).
\end{array}\right.
\end{align*}
$M_{\Upsilon}f(x)$ and $M_{S_{\alpha,\beta}}f(x)$ are given by
$M_{\Upsilon}f(x)=\sup_{t>0}|(f\ast \Upsilon_t)(x)|,\ M_{S_{\alpha,\beta}}f(x)=\sup_{\Upsilon\in S_{\alpha,\beta}(\RR^n)}M_{\Upsilon}f(x).$
And non-tangential maximal functions $(f\ast\Upsilon)_\bigtriangledown(x)$ is defined as usual: $(f\ast\Upsilon)_\bigtriangledown(x)=\sup_{|x-y|\leq t}|(f\ast \Upsilon_t)(y)|.$
The even larger tangential variant $M_{\Upsilon N}^{\ast\ast}$ depending on a parameter N is given by:
$$M_{\Upsilon N}^{\ast\ast}f(x)=\sup_{v\in\RR^n,t>0}\left|\int_{\RR^n}f(u)\frac{1}{t^n}\Upsilon\left(\frac{x-u-v}{t}\right)\left(1+\frac{|v|}{t}\right)^{-N}du \right|.$$
Let $f$ to be a distribution, Hardy spaces $H^p(\RR^n)$ are (c.f.\cite{Du2}):  $\|f\|_{H^p(\RR^n)}= \|M_{S_{\alpha,\beta}}f\|_{L^p(\RR^n)}$, for $0<p<\infty$ with appropriate $\alpha$ and $\beta$ depending on $p$. It is known that $H^p=L^p$ for $p>1$: $\|f\|_{H^p(\RR^n)}= \|f\|_{L^p(\RR^n)}.$\\
In this paper, the Kakeya type maximal function $M_{\delta S_{\alpha,\beta}} f(x)$ is  given by
\begin{eqnarray}\label{1OO}
M_{\delta S_{\alpha,\beta}} f(x)= \sup_{t>0, A\in O(\RR^n), \Upsilon\in S_{\alpha,\beta}(\RR^n)}\left|\int f(x-y)\Upsilon_{I_t}(yA^{-1})dy\right|,
\end{eqnarray}
where $\Upsilon_{AI_t}(y)=\Upsilon_{I_t}(yA^{-1})=\frac{1}{t^n}\Upsilon_{I}\left(\frac{yA^{-1}}{t}\right).$
For some  fixed $t>0$, $M^t_{\delta S_{\alpha,\beta}} f(x)$ can be defined by
\begin{eqnarray}\label{2OO}
M^t_{\delta S_{\alpha,\beta}} f(x)= \sup_{ A\in O(\RR^n), \Upsilon\in S_{\alpha,\beta}(\RR^n)}\left|\int f(x-y)\Upsilon_{I_t}(yA^{-1})dy\right|.
\end{eqnarray}

\begin{lemma}\cite{Du2}\label{oooh1}
For any $\psi\in S(\RR^n)$,  $1<p<\infty$, $f\in L^p(\RR^n)$, $\displaystyle{ N>\frac{n}{p}}$, we could obtain:
$$\|M_{\psi N}^{\ast\ast}f\|_p\lesssim_{N, p}   \|(f\ast\psi)_\bigtriangledown\|_p \lesssim_{p,\psi}   \|f\|_p.$$

\end{lemma}

\begin{lemma}\cite{Du5}\label{oooh2}
Let   $0<C_0<\infty$ and $0<r<\infty$. Then there exist constants $C_1$ and $C_2$(that depend only on n, $C_0$ and r) such that for all $t>0$ and  for all $C^1(\RR^n)$ functions u on $\RR^n$ whose Fourier transform is supported in the ball $|\xi|_e\leq C_0t$ and that satisfies $|u(z)|\leq B(1+|z|_e)^{n/r}$ for some $B>0$, we have the estimation
$$\sup_{z\in\RR^n}\frac{1}{t}\frac{|\nabla u(x-z)|}{(1+t|z|_e)^{n/r}}\leq C_1\sup_{z\in\RR^n}\frac{| u(x-z)|}{(1+t|z|_e)^{n/r}}\leq C_2 \left(M(|u|^r)(x)\right)^{1/r},$$
 where M denotes the Hardy-Littlewood maximal operator.(The constants $C_1$ and $C_2$ are independent of B and u.)

\end{lemma}
\begin{lemma}\label{11}[Phragmen-Lindel\"{o}f Lemma]
Let F be analytic in the open strip $S=\{z\in\CC: 0<Rez<1\}$, continuous and bounded on its closure, such that $|F(z)|\leq C_0$ when
$Rez =0$ and $|F(z)|\leq C_1$ when $Rez =1$. Then $|F(z)|\leq C_0^{1-\theta}C_1^{\theta}$ when $Rez= \theta$ for any $0<\theta<1$.
\end{lemma}

\section{The Case When $2\leq\delta^{-\varepsilon}$}
When $1\leq\delta^{-\varepsilon}\leq2$,  the case that $\delta\sim_{\varepsilon}1$ is trival for the Kakeya type inequalities, thus we only want to discuss the case when $0<\delta\ll1$.
In the following of this  paper, we will discuss under the assumption  that $2\leq\delta^{-\varepsilon}$.

\subsection{Decomposition of the Phase Space}
In this section, we will  decompose $\RR^n$ into a collection of regions:$$\{\xi\in\RR^n: \delta^{-(k-1)\varepsilon}\leq|\xi|_e\leq\delta^{-(k+3)\varepsilon} \}_{k\geq1, k\in\ZZ}\,\hbox{and}\ \ \{\xi\in\RR^n: |\xi|_e\leq1 \}.$$
Then we will give a decomposition of the region $\{\xi\in\RR^n: |\xi|_e\leq1 \}\bigcup\{\xi\in\RR^n: 1\leq|\xi|_e\leq\delta^{-4\varepsilon}\}$ into a collection of smaller ones:
$$\{\xi\in\RR^n: 2^{k-1}\leq|\xi|_e\leq2^{k+3} \}_{s\geq k\geq1, k\in\ZZ}\,\hbox{and}\ \ \{\xi\in\RR^n: |\xi|_e\leq1 \}.$$
Let the functions $\{\widehat{\Phi_k}(\xi)\}_k$ for $k\in\ZZ, k\geq0$ to be defined as:
\begin{align*}\left\{\begin{array}{lll}
\widehat{\Phi_0}(\xi)=\widehat{\varphi}(\xi),\ \ \ \ & &\Phi_0(x)=\varphi(x),\\\\
\widehat{\Phi_k}(\xi)=\widehat{\varphi}(2^{-k}\xi)-\widehat{\varphi}(2^{1-k}\xi),\ \ \ & &\Phi_k(x) = \varphi_{2^{-k}}(x)-\varphi_{2^{-(k-1)}}(x),\ \ \hbox{for}\  k\geq1.
\end{array}\right.
\end{align*}
Then the functions $\{\widehat{\mathbf{\Phi}_k}(\xi)\}_k$ for $k\in\ZZ, k\geq0$ are given by:
\begin{align*}\left\{\begin{array}{lll}
\widehat{\mathbf{\Phi}_0}(\xi)=\widehat{\Phi_0}(\delta\xi_{1},\delta\xi_{2},\ldots,\delta\xi_{n-1}, \xi_n),\ \ \   \mathbf{\Phi_0}(x)=(\Phi_0)_I(x),\\\\
\widehat{\mathbf{\Phi}_k}(\xi) = \widehat{\Phi_k}(\delta\xi_{1},\delta\xi_{2},\ldots,\delta\xi_{n-1}, \xi_n),\ \ \  \mathbf{\Phi_k}(x)=(\Phi_k)_I(x),\ \ \hbox{for}\  k\geq1.
\end{array}\right.
\end{align*}
Thus it is clear that
$supp\, \widehat{\Phi_k}(\xi)\subseteq\{\xi\in \RR^n: 2^{k-1}\leq|\xi|_e\leq2^{k+1}\} ,\ \ \hbox{for}\  k\geq1$
and
$supp\, \widehat{\mathbf{\Phi}_k}(\xi)\subseteq\{\xi\in \RR^n: 2^{k-1}\leq\left((\delta\xi_1)^2+\ldots+(\delta\xi_{n-1})^2+(\xi_n)^2\right)^{1/2}\leq2^{k+1}\}\ \ \hbox{for}\  k\geq1 .$
Also we could deduce that:
$$\sum_{k=0}^{\infty}\widehat{\mathbf{\Phi}_k}(\xi)=1,\ \ \ \hbox{with}\, \mathbf{\Phi_k}(x)=(\Phi_k)_I(x).$$
In the same way, we could define the functions $\{\Psi_k(x)\}_k$ and $\{\mathbf{\Psi}_k(x)\}_k$  for $k\in\ZZ, k\geq0$  as:
\begin{align*}\left\{\begin{array}{lll}
\widehat{\Psi_0}(\xi)=\widehat{\varphi}(\xi),\ \ & &\Psi_0(x)=\varphi(x),\\\\
\widehat{\Psi_k}(\xi)=\widehat{\varphi}(\delta^{k\varepsilon}\xi)-\widehat{\varphi}(\delta^{(k-1)\varepsilon}\xi),\ & &\Psi_k(x) = \varphi_{\delta^{k\varepsilon}}(x)-\varphi_{\delta^{(k-1)\varepsilon}}(x),\ \hbox{for}\  k\geq1,
\end{array}\right.
\end{align*}
\begin{align*}\left\{\begin{array}{lll}
\widehat{\mathbf{\Psi}_0}(\xi)=\widehat{\Psi_0}(\delta\xi_{1},\delta\xi_{2},\ldots,\delta\xi_{n-1}, \xi_n),\ \ \ \mathbf{\Psi_0}(x)=(\Psi_0)_I(x)\\\\
\widehat{\mathbf{\Psi}_k}(\xi) = \widehat{\Psi_k}(\delta\xi_{1},\delta\xi_{2},\ldots,\delta\xi_{n-1}, \xi_n),\ \ \ \mathbf{\Psi_k}(x)=(\Psi_k)_I(x),\ \ \hbox{for}\  k\geq1.
\end{array}\right.
\end{align*}
Then we could deduce that
$supp\, \widehat{\Psi_k}(\xi)\subseteq\{\xi\in \RR^n: \delta^{-(k-1)\varepsilon}\leq|\xi|_e\leq\delta^{-(k+3)\varepsilon}\} \ \ \hbox{for}\  k\geq1$
and
$supp\, \widehat{\mathbf{\Psi}_k}(\xi)\subseteq\{\xi\in \RR^n: \delta^{-(k-1)\varepsilon}\leq\left((\delta\xi_1)^2+\ldots+(\delta\xi_{n-1})^2+(\xi_n)^2\right)^{1/2}\leq\delta^{-(k+3)\varepsilon}\} $
for $k\geq1$ hold.
Thus we could have
$$\sum_{k=0}^{\infty}\widehat{\mathbf{\Psi}_k}(\xi)=1\,\,\,\hbox{with}\,\mathbf{\Psi_k}(x)=(\Psi_k)_I(x).$$
Notice that $\delta^{1+(k+3)\varepsilon}|\xi|_e\leq1$ for $\xi\in supp\,\widehat{\Psi_k}(\xi)$. Then we could obtain:
$$\widehat{\varphi}(\delta^{1+(k+3)\varepsilon}\xi)=1,\ \ \hbox{for}\ \xi\in supp\,\widehat{\Psi_k}(\xi),$$
and
$$\mathfrak{F}(\Upsilon_I)(\xi)=\sum_{k=0}^{\infty}\frac{\widehat{\mathbf{\Psi}_k}(\xi)}{\widehat{\varphi}(\delta^{1+(k+3)\varepsilon}\xi)}\mathfrak{F}(\Upsilon_I)(\xi)\widehat{\varphi}(\delta^{1+(k+3)\varepsilon}\xi).$$
We set $\widehat{\eta_1^k}(\xi)$,\,($\hbox{for}\,0\leq k, k\in\ZZ$) as:
$$\widehat{\eta_1^k}(\xi)=\frac{\widehat{\mathbf{\Psi}_k}(\xi)}{\widehat{\varphi}(\delta^{1+(k+3)\varepsilon}\xi)}\mathfrak{F}(\Upsilon_I)(\xi)=\widehat{\mathbf{\Psi}_k}(\xi)\mathfrak{F}(\Upsilon_I)(\xi).\,\ \ \ \hbox{for}\,0\leq k,\  k\in\ZZ.$$
It is easy to see that $\widehat{\eta_1^k}(\xi)\,\hbox{and}\,\eta_1^k(x)\in S(\RR^n)$.
$\exists s\in\NN$, such that $$\delta^{-4\varepsilon}\sim 2^{s},\ \ \ \hbox{and}\,\sum_{k=0}^{s}\,\widehat{\mathbf{\Phi}_k}(\xi)=1\,\hbox{for}\,\xi\in \,supp\widehat{\mathbf{\Psi}_0}\bigcup\,supp\widehat{\mathbf{\Psi}_1}.$$
We set $\widehat{\eta_0^k}(\xi)$ ($\,k=0,1,2,\cdots,s$) as:
$$\widehat{\eta_0^k}(\xi)=\frac{\left(\widehat{\mathbf{\Psi}_0}(\xi)+\widehat{\mathbf{\Psi}_1}(\xi)\right)\widehat{\mathbf{\Phi}_k}(\xi)}{\widehat{\varphi}(2^{-(k+1)}\delta\xi)}\mathfrak{F}(\Upsilon_I)(\xi).\,\ \ \ \hbox{for}\,k=0,1,2,\cdots,s.$$
Notice that $2^{-(k+1)}\delta|\xi|_e\leq1$ holds, when $\xi\in supp\,\widehat{\eta_0^k}(\xi)$. Thus we could obtain:
$$\widehat{\varphi}(2^{-(k+1)}\delta\xi)=1\ \ \hbox{when}\ \xi\in supp\,\widehat{\eta_0^k}(\xi).$$
Thus
$$\widehat{\eta_0^k}(\xi)=\left(\widehat{\mathbf{\Psi}_0}(\xi)+\widehat{\mathbf{\Psi}_1}(\xi)\right)\widehat{\mathbf{\Phi}_k}(\xi)\mathfrak{F}(\Upsilon_I)(\xi)\,\ \ \ \hbox{for}\,k=0,1,2,\cdots,s.$$
Thus we could  write $\mathfrak{F}(\Upsilon_I)(\xi)$ and $\mathfrak{F}(\Upsilon_{AI})(\xi)\,(\widehat{\varphi}$ is radial, $A$ is a variable (not fixed) matrix with $ A\in O(\RR^n)$) as
\begin{eqnarray}\label{36}
\mathfrak{F}(\Upsilon_I)(\xi)=\sum_{k=0}^s \widehat{\eta_0^k}(\xi)\widehat{\varphi}(2^{-(k+1)}\delta\xi)+  \sum_{k=2}^\infty\widehat{\eta_1^k}(\xi)\widehat{\varphi}(\delta^{1+(k+3)\varepsilon}\xi)
 \end{eqnarray}
 \begin{eqnarray}\label{26}
 \mathfrak{F}(\Upsilon_{AI})(\xi)=\mathfrak{F}(\Upsilon_I)(A\xi)=\sum_{k=0}^s \widehat{\eta_0^k}(A\xi)\widehat{\varphi}(2^{-(k+1)}\delta\xi)+   \sum_{k=2}^\infty\widehat{\eta_1^k}(A\xi)\widehat{\varphi}(\delta^{1+(k+3)\varepsilon}\xi)
\end{eqnarray}
where $2^s\thicksim\delta^{-4\varepsilon}$.

\subsection{Two Lemmas}
In this section, we will estimate the  integrals (in Lemma\,\ref{2} and Lemma\,\ref{6})  associated with   $\eta^k_0(x)$ and $\eta^k_1(x)$  given in Formulas\,(\ref{36},\,\ref{26}).

\begin{lemma}\label{2}
For $N\geq0,\  N\in\RR$, $k\in\NN$, $k\geq2$, $\Upsilon\in S_{\alpha, \beta}(\RR^n)$ with appropriate $\alpha, \beta$ depending on $\varepsilon,\ n,\ N$, we have
$$\int_{\RR^n}(1+\delta^{-(k+3)\varepsilon}\delta^{-1}|x|_e)^N|\eta^k_1(x)|dx\lesssim_{ N,n,\varphi,\varepsilon} \delta^{k\varepsilon}. $$
\end{lemma}

\begin{proof}
First we will prove that for $l\in\RR$, $l\geq0$, $k\in\NN$, $k\geq2$, the following inequality holds:
\begin{eqnarray}\label{3}
|x|_e^{l+2n}|\eta^k_1(x)|\lesssim_{l,n,\varphi,\varepsilon}\delta^{l+k\varepsilon+(k+3)l\varepsilon}.
\end{eqnarray}
Notice that the following inequality holds for $0<\delta<1$, for any $ m\in\ZZ,\,m\geq0$:
\begin{eqnarray}\label{4}
|x|_e^{2m+2n}|\eta^k_1(x)|\leq \left(\left(\frac{x_1}{\delta}\right)^2+\left(\frac{x_1}{\delta}\right)^2+\ldots+\left(\frac{x_{n-1}}{\delta}\right)^2+x_n^2\right)^{m+n}|\eta^k_1(x)|.
\end{eqnarray}
Thus by the formula of integration by parts, we could deduce the following for any $ m\in\ZZ,\,m\geq0$:
\begin{eqnarray}\label{9}
&&\left(\left(\frac{x_1}{\delta}\right)^2+\left(\frac{x_1}{\delta}\right)^2+\ldots+\left(\frac{x_{n-1}}{\delta}\right)^2+x_n^2\right)^{m+n}|\eta^k_1(x)|
\\ \nonumber &=&\left|\int_{\RR^n}C \left(\left(\frac{\partial_{\xi_1}}{\delta}\right)^2+\left(\frac{\partial_{\xi_2}}{\delta}\right)^2+\ldots+\left(\frac{\partial_{\xi_{n-1}}}{\delta}\right)^2+\partial_{\xi_{n}}^2\right)^{m+n}\widehat{\eta_1^k}(\xi)e^{2\pi i <x, \xi>} d\xi\right|.
\end{eqnarray}
Make a variable substitution:
$$(\delta\xi_1, \delta\xi_2\ldots\delta\xi_{n-1}, \xi_n)\rightarrow(\xi'_1, \xi'_2\ldots\xi'_{n-1}, \xi'_n).$$
We could write Formula\,(\ref{9}) as:
\begin{eqnarray}\label{10}
\left(\left(\frac{x_1}{\delta}\right)^2+\left(\frac{x_1}{\delta}\right)^2+\ldots+\left(\frac{x_{n-1}}{\delta}\right)^2+x_n^2\right)^{m+n}|\eta^k_1(x)|
\\ \nonumber=\frac{1}{\delta^{n-1}}\left|\int_{\RR^n}C \left((\triangle_{\xi'})^{n+m}\widehat{\eta_1^k}(\xi')\right)e^{2\pi i <x, \xi>} d\xi'\right|,
\end{eqnarray}
where $\triangle_{\xi'}$ is the Laplace Operator: $\triangle_{\xi'}=\partial_{\xi'_1}^2+ \partial_{\xi'_2}^2+\cdots+ \partial_{\xi'_n}^2$. We could also deduce that
$$(\triangle_{\xi'})^{n+m}\widehat{\eta_1^k}(\xi')=(\triangle_{\xi'})^{n+m}\left(\widehat{\Psi_k}(\xi')\mathfrak{F}(\Upsilon)(\xi')\right).$$
Thus $\left((\triangle_{\xi'})^{n+m}\widehat{\eta_1^k}(\xi')\right)\in S(\RR^n)$, $\left||\xi'|_e^{|\alpha'|}(\triangle_{\xi'})^{\beta'}\widehat{\eta_1^k}(\xi') \right|\lesssim_{\alpha', \beta'}1,\ \hbox{for\,appropriate\,}\alpha', \beta',$ and $$supp\,\left((\triangle_{\xi'})^{n+m}\widehat{\eta_1^k}(\xi')\right)\subseteq\{\xi'\in \RR^n: \delta^{-(k-1)\varepsilon}\leq|\xi'|_e\leq\delta^{-(k+3)\varepsilon}\} \ \ \hbox{for}\  k\geq2.$$
When $\delta^{-(k-1)\varepsilon}\leq|\xi'|_e\leq\delta^{-(k+3)\varepsilon}$, $k\geq2$, $\delta^{-\varepsilon}\geq2$, we could deduce that
$$ \delta^{-\frac{k}{2}\varepsilon}\leq\delta^{-(k-1)\varepsilon}\leq|\xi'|_e\leq\delta^{-(k+3)\varepsilon}\leq\delta^{-3k\varepsilon}.$$
That is
$$ |\xi'|_e^{\frac{1}{3}}\leq\delta^{-k\varepsilon}\leq|\xi'|_e^{2},\,\ \ \
\delta^{-\varepsilon}\leq\delta^{-k\varepsilon}\leq|\xi'|_e^{2}.$$
Then we could deduce that
\begin{eqnarray}\label{5}
|x|_e^{2m+2n}|\eta^k_1(x)|&\lesssim&\frac{1}{\delta^{n-1}}\int_{\RR^n}\left| \left((\triangle_{\xi'})^{n+m}\widehat{\eta_1^k}(\xi')\right)\right| d\xi'
\\ \nonumber &\lesssim&\delta^{2m+k\varepsilon+2(k+3)m\varepsilon}\int_{\RR^n} |\xi'|_e^{\frac{2n+4m}{\varepsilon}+8m}\left| \left((\triangle_{\xi'})^{n+m}\widehat{\eta_1^k}(\xi')\right)\right| d\xi'
\\ \nonumber &\lesssim&_{m,n,\varphi,\varepsilon}\delta^{2m+k\varepsilon+2(k+3)m\varepsilon}.
\end{eqnarray}
Thus similar to Formula\,(\ref{5}), we could obtain
\begin{eqnarray}\label{6**}
|x|_e^{2m}|\eta^k_1(x)|&\lesssim&\frac{1}{\delta^{n-1}}\int_{\RR^n}\left| \left((\triangle_{\xi'})^{n+m}\widehat{\eta_1^k}(\xi')\right)\right| d\xi'
\\ \nonumber &\lesssim&_{m,n,\varphi,\varepsilon}\delta^{2m+k\varepsilon+2(k+3)m\varepsilon},
\end{eqnarray}
where $k\geq2$, $m\in\ZZ,\,m\geq0$. By Lemma\,\ref{11} and Formula\,(\ref{5}), we could deduce Formula\,(\ref{3}). Thus we could obtain the following  inequality for $N\geq 0,\  N\in\RR$
\begin{eqnarray}\label{5***}
\left(\frac{\delta^{-(k+3)\varepsilon}}{\delta}|x|_e\right)^N|\eta^k_1(x)|\lesssim_{N,n,\varphi,\varepsilon}\delta^{k\varepsilon}\frac{1}{|x|_e^{2n}}.
\end{eqnarray}
By Lemma\,\ref{11} and Formula\,(\ref{6**}),  for $l\in\RR$, $l\geq0$, $k\in\NN$, $k\geq2$, the following inequality holds:
\begin{eqnarray}\label{3*}
|x|_e^{l}|\eta^k_1(x)|\lesssim_{l,n,\varphi,\varepsilon}\delta^{l+k\varepsilon+(k+3)l\varepsilon}.
\end{eqnarray}
Then we could obtain the following  inequality for $N\geq0,\  N\in\RR$, $k\in\NN$, $k\geq2$,
\begin{eqnarray}\label{5****}
\left(\frac{\delta^{-(k+3)\varepsilon}}{\delta}|x|_e\right)^N|\eta^k_1(x)|\lesssim_{N,n,\varphi,\varepsilon}\delta^{k\varepsilon}.
\end{eqnarray}
By Formulas\,(\ref{5***},\,\ref{5****}), we could deduce that for $N\geq 0,\  N\in\RR$, $k\in\NN$, $k\geq2$, the following Formulas\,(\ref{7**},\,\ref{6*}) hold:
\begin{eqnarray}\label{7**}
\int_{\RR^n}|\eta^k_1(x)|dx\lesssim_{ n,\varphi,\varepsilon} \delta^{k\varepsilon}. \end{eqnarray}
and
\begin{eqnarray}\label{6*}
\int_{\RR^n}(\delta^{-(k+3)\varepsilon}\delta^{-1}|x|_e)^N|\eta^k_1(x)|dx\lesssim_{ N,n,\varphi,\varepsilon}\delta^{k\varepsilon}. \end{eqnarray}
Then we could obtain the Lemma\,\ref{2} directly from Formula\,(\ref{6*},\,\ref{7**}). This proves the Lemma.$\hfill\blacksquare$
\end{proof}

\begin{lemma}\label{6}
For $N\geq 0, N\in\RR$,  $k\in\NN$, $0\leq k\leq s$ where $2^s\thicksim\delta^{-4\varepsilon}$, $\Upsilon\in S_{\alpha, \beta}(\RR^n)$ with appropriate $\alpha, \beta$ depending on $\varepsilon,\ n,\ N$, the following two inequalities hold:
\begin{eqnarray}\label{7}
\int_{\RR^n}(1+2^{k+1}\delta^{-1}|x|_e)^N|\eta^k_0(x)|dx\lesssim_{N,n,\varphi,\varepsilon} \delta^{-4(N+1)\varepsilon}\delta^{-N}2^{-k}, \end{eqnarray}
\begin{eqnarray}\label{8}
\int_{\RR^n}(1+2^{k+1}|x|_e)^N|\eta^k_0(x)|dx\lesssim_{N,n,\varphi,\varepsilon} \delta^{-4(N+1)\varepsilon}2^{-k}.
\end{eqnarray}

\end{lemma}
\begin{proof}
Notice that $2^s\thicksim\delta^{-4\varepsilon}$, thus for any $k\in\{0,1,\ldots, s\}$ and $N\geq 0, N\in\RR$, we have
\begin{eqnarray}\label{21}
1\lesssim\left(\frac{1}{2^{k+1}\delta^{4\varepsilon}}\right)^{N}.
\end{eqnarray}
Notice that $2^{-(k+1)}\delta|\xi|_e\leq1$ holds, when $\xi\in supp\,\widehat{\eta_0^k}(\xi)$. Thus we could obtain:
$$\widehat{\varphi}(2^{-(k+1)}\delta\xi)=1\ \ \hbox{when}\ \xi\in supp\,\widehat{\eta_0^k}(\xi).$$
Then we could write $\widehat{\eta_0^k}(\xi)$ as:
$$\widehat{\eta_0^k}(\xi)=\left(\left(\widehat{\mathbf{\Psi}_0}(\xi)+\widehat{\mathbf{\Psi}_1}(\xi)\right)\widehat{\mathbf{\Phi}_k}(\xi)\right)\mathfrak{F}(\Upsilon_I)(\xi).$$
It is clear that the following Formulas\,(\ref{12},\,\ref{13},\,\ref{14},\,\ref{15},\,\ref{16}) hold:
\begin{eqnarray}\label{12}
\left(\partial_{\xi}^{\alpha_1}\widehat{\mathbf{\Psi}_0}(\xi)\right)^{\vee}(x)=(-2\pi i x)^{\alpha_1}\varphi_I(x)
\end{eqnarray}
\begin{eqnarray}\label{13}
\left(\partial_{\xi}^{\alpha_1}\widehat{\mathbf{\Psi}_1}(\xi)\right)^{\vee}(x)=(-2\pi i x)^{\alpha_1}\left(\frac{1}{\delta^\varepsilon}\right)^n\varphi_I\left(\frac{x}{\delta^\varepsilon}\right)-(-2\pi i x)^{\alpha_1}\varphi_I(x)
\end{eqnarray}

\begin{eqnarray}\label{14}
& &\left(\partial_{\xi}^{\beta_1}\widehat{\mathbf{\Phi}_k}(\xi)\right)^{\vee}(x)\\ \nonumber&=&(-2\pi i x)^{\beta_1}2^{kn}\varphi_I(2^{k}x)-(-2\pi i x)^{\beta_1}2^{(k-1)n}\varphi_I(2^{(k-1)}x)\,\ \ (\hbox{for}\,k\geq1),
\end{eqnarray}

\begin{eqnarray}\label{15}
\left(\partial_{\xi}^{\beta_1}\widehat{\mathbf{\Phi}_0}(\xi)\right)^{\vee}(x)=(-2\pi i x)^{\beta_1}\varphi_I(x),
\end{eqnarray}
\begin{eqnarray}\label{16}
\left(\partial_{\xi}^{\gamma_1}\mathfrak{F}(\Upsilon_I)(\xi)\right)^{\vee}(x)=(-2\pi i x)^{\gamma_1}\Upsilon_I(x).
\end{eqnarray}
By Young Inequality  we could have
\begin{eqnarray}\label{17}
\int|\eta_0^k(x)|dx&\leq&\|\left(\mathbf{\Psi}_0+\mathbf{\Psi}_1\right)\ast\mathbf{\Phi}_k\ast \Upsilon_I\|_1
\\&\leq& \nonumber\|\left(\mathbf{\Psi}_0+\mathbf{\Psi}_1\right)\|_1\|\mathbf{\Phi}_k\|_1\|\Upsilon_I\|_1
\\ \nonumber &\lesssim_{\varphi}& 1.
\end{eqnarray}
By the Formula of Integration by Parts, we could deduce the following for any $ m\in\NN$:
\begin{eqnarray}\label{18}
& &|x|_e^{2n+2m}|\eta_0^k(x)|\\&=&\nonumber\left|\int_{\RR^n}C \left((\triangle_{\xi})^{n+m}\widehat{\eta_0^k}(\xi)\right)e^{2\pi i <x, \xi>} d\xi\right|
\\ \nonumber&=&\left|\sum_{|\alpha_1|+|\beta_1|+|\gamma_1|=2m+2n}\left(\partial_{\xi}^{\alpha_1}\widehat{\mathbf{\Psi}_1}(\xi)+\partial_{\xi}^{\alpha_1}\widehat{\mathbf{\Psi}_0}(\xi)\right)^{\vee}\ast\left(\partial_{\xi}^{\beta_1}\widehat{\mathbf{\Phi}_k}(\xi)\right)^{\vee}\ast\left(\partial_{\xi}^{\gamma_1}\mathfrak{F}(\Upsilon_I)(\xi)\right)^{\vee}(x) \right|.
\end{eqnarray}
By Young Inequality, Formula\,(\ref{18}) and Formulas\,(\ref{12},\,\ref{13},\,\ref{14},\,\ref{15},\,\ref{16}), we could obtain
\begin{eqnarray}\label{19}
& & \int||x|_e^{2n+2m}\eta_0^k(x)|dx\\&\leq&\nonumber\sum_{|\alpha_1|+|\beta_1|+|\gamma_1|=2m+2n}\bigg\|\left(\partial_{\xi}^{\alpha_1}\widehat{\mathbf{\Psi}_1}(\xi)+\partial_{\xi}^{\alpha_1}\widehat{\mathbf{\Psi}_0}(\xi)\right)^{\vee}\bigg\|_1\bigg\|\left(\partial_{\xi}^{\beta_1}\widehat{\mathbf{\Phi}_k}(\xi)\right)^{\vee}\bigg\|_1\bigg\|\left(\partial_{\xi}^{\gamma_1}\mathfrak{F}(\Upsilon_I)(\xi)\right)^{\vee}\bigg\|_1 \\  &\lesssim&_{\varphi,n,m} 1\ \ \hbox{for}\ m\in\NN\ \ \nonumber.
\end{eqnarray}
By Lemma\,\ref{11} and Formula\,(\ref{19},\,\ref{17}), we could deduce the following Formula\,(\ref{20}).
\begin{eqnarray}\label{20}
\int||x|_e^{l}\eta_0^k(x)|dx\lesssim_{\varphi,l} 1\ \ \hbox{for}\ l\in\RR, l\geq0.
\end{eqnarray}
By Formulas\,(\ref{17},\,\ref{20}), the following two inequalities hold for $N\geq0$, $k\in\{0,1,\ldots, s\}$:
\begin{eqnarray}\label{7*}
\int_{\RR^n}(1+2^{k+1}\delta^{-1}|x|_e)^N|\eta^k_0(x)|dx\lesssim_{N,n,\varphi,\varepsilon} \delta^{-N}2^{(k+1)N},
\end{eqnarray}
\begin{eqnarray}\label{8*}
\int_{\RR^n}(1+2^{k+1}|x|_e)^N|\eta^k_0(x)|dx\lesssim_{N,n,\varphi,\varepsilon} 2^{(k+1)N}.
\end{eqnarray}
From Formulas\,(\ref{21},\,\ref{7*},\,\ref{8*}), we could obtain the Formula\,(\ref{7}) and Formula\,(\ref{8}) together. This proves the Lemma.$\hfill\blacksquare$
\end{proof}

From Lemma\,\ref{2} and Lemma\,\ref{6}, we could obtain the following inequalities\,(\ref{22},\,\ref{23},\,\ref{24}). For $N\geq 0, N\in\RR$, $k\in\NN$, $k\geq2$, we have
\begin{eqnarray}\label{22}
\int_{\RR^n}(1+\delta^{-(k+3)\varepsilon}\delta^{-1}|x|_e)^N|\eta^k_1(xA^{-1})|dx\lesssim_{N,n,\varphi,\varepsilon} \delta^{k\varepsilon},
\end{eqnarray}
where $A$ is a variable (not fixed) matrix with $ A\in O(\RR^n)$.
For  $N\geq 0, N\in\RR$, $k\in\NN$, $0\leq k\leq s$ where $2^s\thicksim\delta^{-4\varepsilon}$, we have Formulas\,(\ref{23},\,\ref{24})
\begin{eqnarray}\label{23}
\int_{\RR^n}(1+2^{k+1}\delta^{-1}|x|_e)^N|\eta^k_0(xA^{-1})|dx\lesssim_{N,n,\varphi,\varepsilon} \delta^{-4(N+1)\varepsilon}\delta^{-N}2^{-k},
\end{eqnarray}
\begin{eqnarray}\label{24}
\int_{\RR^n}(1+2^{k+1}|x|_e)^N|\eta^k_0(xA^{-1})|dx\lesssim_{N,n,\varphi,\varepsilon} \delta^{-4(N+1)\varepsilon}2^{-k},
\end{eqnarray}
where $A$ is a variable (not fixed) matrix with $ A\in O(\RR^n)$.

\subsection{MAIN RESULTS}
From Formulas\,(\ref{22},\,\ref{23},\,\ref{24}), we will obtain our main results in this section:

\begin{proposition}\label{25}
For $p>1$ with appropriate $\alpha$, $\beta$ depending on $\varepsilon,\ n,\ p$,  we have
\begin{eqnarray}\label{41}
\big\|M_{\delta S_{\alpha,\beta}} f \big\|_p\lesssim_{p,n,\varphi,\varepsilon}\left(\frac{1}{\delta}\right)^{4(\frac{n}{p}+2)\varepsilon}\big\|(f\ast\varphi_{\delta})_\bigtriangledown\big\|_p,
\end{eqnarray}
and
\begin{eqnarray}\label{42}
\big\|M_{\delta S_{\alpha,\beta}} f \big\|_p\lesssim_{p,n,\varphi,\varepsilon}\left(\frac{1}{\delta}\right)^{\frac{n}{p}+4(\frac{n}{p}+3)\varepsilon}\|f\|_p,
\end{eqnarray}
where $(f\ast\varphi_{\delta})_\bigtriangledown(x)=\sup_{|x-y|\leq t}|(f\ast(\varphi_{\delta}))_t(y)|.$

\end{proposition}

\begin{proof}  Notice that $f\in L^p(\RR^n)$ is a distribution.  By Formula\,(\ref{26}),  we could obtain:
\begin{eqnarray}\label{27}
& &|M_{\delta S_{\alpha,\beta}} f(x)| \\ \nonumber&=&\sup_{t>0, A\in O(\RR^n), \Upsilon\in S_{\alpha, \beta}(\RR^n)}\left|\int f(x-y)\Upsilon_{I_t}(yA^{-1})dy\right|
\\ \nonumber &\leq&\sum_{k=2}^{\infty}\sup_{t>0, A\in O(\RR^n)}\left|\int_{\RR^n} f(x-y) \int_{\RR^n}t^{-n} \eta^k_1((uA^{-1})/t)\varphi_{\delta^{(k+3)\varepsilon}\delta t}(y-u)dudy\right|
\\ \nonumber &+&\sum_{k=0}^{s}\sup_{t>0, A\in O(\RR^n)}\left|\int_{\RR^n} f(x-y) \int_{\RR^n}t^{-n} \eta_0^k((uA^{-1})/t)\varphi_{2^{-(k+1)}\delta t}(y-u)dudy\right|
\\  \nonumber&\leq&\sum_{k=2}^{\infty}\sup_{t>0, A\in O(\RR^n)}\left|\int_{\RR^n}t^{-n} \eta_1^k((uA^{-1})/t) f\ast\varphi_{\delta^{(k+3)\varepsilon}\delta t}(x-u)du\right|
\\  \nonumber&+&\sum_{k=0}^{s}\sup_{t>0, A\in O(\RR^n)}\left|\int_{\RR^n}t^{-n} \eta_0^k((uA^{-1})/t) f\ast\varphi_{2^{-(k+1)}\delta t}(x-u)du\right|
\\  \nonumber&\leq&\sum_{k=2}^{\infty}\sup_{t>0, A\in O(\RR^n)}\left|M_{\varphi_\delta N}^{\ast\ast}f(x)\int_{\RR^n}t^{-n} \left|\eta_1^k((uA^{-1})/t)\right|\left(1+\frac{|u|_{e}}{\delta^{(k+3)\varepsilon} t}\right)^{N}du\right|
\\  \nonumber&+&\sum_{k=0}^{s}\sup_{t>0, A\in O(\RR^n)}\left|M_{\varphi_\delta N}^{\ast\ast}f(x)\int_{\RR^n}t^{-n} \left|\eta_0^k((uA^{-1})/t)\right|\left(1+\frac{|u|_{e}}{2^{-(k+1)}t}\right)^{N}du\right|,
\end{eqnarray}
where $2^s\thicksim\delta^{-4\varepsilon}$.
Lemma\,\ref{oooh1} and Formulas\,(\ref{22},\,\ref{24},\,\ref{27}) yield to
\begin{eqnarray}\label{29}
\big\|M_{\delta S_{\alpha,\beta}} f \big\|_p\lesssim_{p,N,n,\varphi,\varepsilon}\left(\frac{1}{\delta}\right)^{4(N+2)\varepsilon}\|(f\ast\varphi_{\delta})_\bigtriangledown\|_p\ \ \hbox{for}\  p>1, N> n/p.
\end{eqnarray}
Similar to Formula\,(\ref{27}), we could also obtain:
\begin{eqnarray}\label{28}
& &\sup_{t>0, A\in O(\RR^n)}\left|\int f(x-y)\Upsilon_{I_t}(yA^{-1})dy\right|
\\  \nonumber&\leq&\sum_{k=2}^{\infty}\sup_{t>0, A\in O(\RR^n)}\left|M_{\varphi N}^{\ast\ast}f(x)\int_{\RR^n}t^{-n} \left|\eta_1^k((uA^{-1})/t)\right|\left(1+\frac{|u|_{e}}{\delta^{(k+3)\varepsilon}\delta t}\right)^{N}du\right|
\\  \nonumber&+&\sum_{k=0}^{s}\sup_{t>0, A\in O(\RR^n)}\left|M_{\varphi N}^{\ast\ast}f(x)\int_{\RR^n}t^{-n} \left|\eta_0^k((uA^{-1})/t)\right|\left(1+\frac{|u|_{e}}{2^{-(k+1)}\delta t}\right)^{N}du\right|.
\end{eqnarray}
Notice that  $2^s\thicksim\delta^{-4\epsilon}$, thus Lemma\,\ref{oooh1} and Formulas\,(\ref{22},\,\ref{23},\,\ref{28}) yield to
\begin{eqnarray}\label{30}
\big\|M_{\delta S_{\alpha,\beta}} f \big\|_p\lesssim_{p,N,n,\varphi,\varepsilon}\left(\frac{1}{\delta}\right)^{N}\left(\frac{1}{\delta}\right)^{4(N+2)\varepsilon}\|f\|_p\ \ \hbox{for}\  p>1, N> n/p, N\in\RR.
\end{eqnarray}
Let $N$ be $N=\frac{n}{p}+\varepsilon$. From Formulas\,(\ref{29},\,\ref{30}), we could prove the Proposition\,\ref{25}.
$\hfill\blacksquare$
\end{proof}

\begin{theorem}\label{35}
For $\infty>p>r>1$, $ 0<t\leq \delta^{-\varepsilon}$, $f(x)\in L^p(\RR^n)$  and $supp\,\hat{f}(\xi)\subseteq B(0,1)$. Then with appropriate $\alpha$, $\beta$ depending on $\varepsilon,\ n,\ r$,  we could obtain:
$$\big\|M^t_{\delta S_{\alpha,\beta}} f \big\|_p\lesssim_{p,n,\varphi,\varepsilon}\left(\frac{1}{\delta}\right)^{5\left(\frac{n}{r}+2\right)\varepsilon}\big\|f\big\|_p.$$
\end{theorem}
\begin{proof}   By Formula\,(\ref{26}), we could write $M^t_{\delta S_{\alpha,\beta}} f$ as
\begin{eqnarray}\label{31}
& &\left|M^t_{\delta S_{\alpha,\beta}} f(x)\right|
\\  \nonumber&\leq&\sum_{k=2}^{\infty}\sup_{ u\in\RR^n}\left|\frac{f\ast\varphi_{\delta^{(k+3)\varepsilon}\delta t}(x-u)}{\left(1+\frac{|u|_{e}}{\delta^{(k+3)\varepsilon} t}\right)^{n/r}}\right|\sup_{ A\in O(\RR^n)}\left|\int_{\RR^n}t^{-n} \left|\eta_1^k((uA^{-1})/t)\right|\left(1+\frac{|u|_{e}}{\delta^{(k+3)\varepsilon} t}\right)^{n/r}du\right|
\\  \nonumber&+&\sum_{k=0}^{s}\sup_{ u\in\RR^n}\left|\frac{f\ast\varphi_{2^{-(k+1)}\delta t}(x-u)}{\left(1+\frac{|u|_{e}}{2^{-(k+1)}t}\right)^{n/r}}\right|\sup_{ A\in O(\RR^n)}\left|\int_{\RR^n}t^{-n} \left|\eta_0^k((uA^{-1})/t)\right|\left(1+\frac{|u|_{e}}{2^{-(k+1)}t}\right)^{n/r}du\right|,
\end{eqnarray}
where $2^s\thicksim\delta^{-4\varepsilon}$.

By Holder Inequality,
$\left|f\ast\varphi_{\delta^{(k+3)\varepsilon}\delta t}(x)\right|$ and $\left|f\ast\varphi_{2^{-(k+1)}\delta t}(x)\right|$ are both bounded functions for any $x\in\RR$, $\delta>0$, $k\in\NN$ and $t>0$. Thus for some $B>0$, we could have
$$\left|f\ast\varphi_{\delta^{(k+3)\varepsilon}\delta t}(x)\right|\leq B(1+|x|_e)^{n/r}\,\,\hbox{and},\,\ \left|f\ast\varphi_{2^{-(k+1)}\delta t}(x)\right|\leq B(1+|x|_e)^{n/r}.$$
It is also clear that $supp\,\mathfrak{F}\left(f\ast\varphi_{\delta^{(k+3)\varepsilon}\delta t}\right)(\xi)\subseteq B(0,1)$, $supp\,\mathfrak{F}\left(f\ast\varphi_{2^{-(k+1)}\delta t}\right)(\xi)\subseteq B(0,1)$. We could also deduce that
$f\ast\varphi_{\delta^{(k+3)\varepsilon}\delta t}(x)\in C^1(\RR^n),\ \   f\ast\varphi_{2^{-(k+1)}\delta t}(x)\in C^1(\RR^n)$.
Thus by Lemma\,\ref{oooh2}, we could obtain
\begin{eqnarray}\label{32}
\sup_{ u\in\RR^n}\left|\frac{f\ast\varphi_{\delta^{(k+3)\varepsilon}\delta t}(x-u)}{\left(1+\frac{|u|_{e}}{\delta^{(k+3)\varepsilon} t}\right)^{n/r}}\right| &\leq&\sup_{u\in\RR^n}\frac{| f\ast\varphi_{\delta^{(k+3)\varepsilon}\delta t}(x-u)|}{(1+|u|_e)^{n/r}}\\ \nonumber &\leq& C_2 (M(|f\ast\varphi_{\delta^{(k+3)\varepsilon}\delta t}|^r)(x))^{1/r},
\end{eqnarray}
and
\begin{eqnarray}\label{33}
\sup_{ u\in\RR^n}\left|\frac{f\ast\varphi_{2^{-(k+1)}\delta t}(x-u)}{\left(1+\frac{|u|_{e}}{2^{-(k+1)}t}\right)^{n/r}}\right| &\leq&\sup_{u\in\RR^n}\frac{| f\ast\varphi_{2^{-(k+1)}\delta t}(x-u)|}{(1+|u|_e)^{n/r}} \\ \nonumber &\leq& C_2 \left(M(|f\ast\varphi_{2^{-(k+1)}\delta t}|^r)(x)\right)^{1/r}.
\end{eqnarray}
 Notice that $2^s\thicksim\delta^{-4\varepsilon}$,  by Formulas\,(\ref{22},\,\ref{24},\,\ref{31},\,\ref{32},\,\ref{33}), we could deduce that for $\infty>p>r>1$, $ 0<t\leq 1$
\begin{eqnarray*}
\big\|M^t_{\delta S_{\alpha,\beta}} f \big\|_p&\lesssim&_{p,n,\varphi,\varepsilon}\sum_{k=2}^{\infty}\delta^{k\varepsilon}\left|\int_{\RR^n}|f\ast\varphi_{\delta^{(k+3)\varepsilon}\delta t}|^p(x)dx\right|^{1/p}
\\&+& \sum_{k=0}^{s}\delta^{-4(\frac{n}{r}+1)\varepsilon}2^{-k}\left|\int_{\RR^n}|f\ast\varphi_{2^{-(k+1)}\delta t}|^p(x)dx\right|^{1/p}
\\&\lesssim&_{p,n,\varphi,\varepsilon}\left(\frac{1}{\delta}\right)^{4\left(\frac{n}{r}+2\right)\varepsilon}\big\|f\big\|_p.
\end{eqnarray*}
When $ 1<t\leq \delta^{-\varepsilon}$, notice that
\begin{eqnarray}\label{32*}
\sup_{ u\in\RR^n}\left|\frac{f\ast\varphi_{2^{-(k+1)}\delta t}(x-u)}{\left(1+\frac{|u|_{e}}{2^{-(k+1)}t}\right)^{n/r}}\right| &\leq& t^{n/r}\sup_{ u\in\RR^n}\left|\frac{f\ast\varphi_{2^{-(k+1)}\delta t}(x-u)}{\left(1+\frac{|u|_{e}}{2^{-(k+1)}}\right)^{n/r}}\right|\\ \nonumber&\leq&t^{n/r}\sup_{u\in\RR^n}\frac{| f\ast\varphi_{2^{-(k+1)}\delta t}(x-u)|}{(1+|u|_e)^{n/r}}\\ \nonumber&\leq &C_2 \delta^{-(n/r)\varepsilon} \left(M(|f\ast\varphi_{2^{-(k+1)}\delta t}|^r)(x)\right)^{1/r}
\end{eqnarray}
and
\begin{eqnarray}\label{33*}
\sup_{ u\in\RR^n}\left|\frac{f\ast\varphi_{\delta^{(k+3)\varepsilon}\delta t}(x-u)}{\left(1+\frac{|u|_{e}}{\delta^{(k+3)\varepsilon}t}\right)^{n/r}}\right| &\leq& t^{n/r}\sup_{ u\in\RR^n}\left|\frac{f\ast\varphi_{\delta^{(k+3)\varepsilon}\delta t}(x-u)}{\left(1+\frac{|u|_{e}}{\delta^{(k+3)\varepsilon}}\right)^{n/r}}\right|\\ \nonumber&\leq&t^{n/r}\sup_{u\in\RR^n}\frac{| f\ast\varphi_{\delta^{(k+3)\varepsilon}\delta t}(x-u)|}{(1+|u|_e)^{n/r}}\\ \nonumber&\leq &C_2 \delta^{-(n/r)\varepsilon} \left(M(|f\ast\varphi_{\delta^{(k+3)\varepsilon}\delta t}|^r)(x)\right)^{1/r}
\end{eqnarray}
hold. From Formulas\,(\ref{22},\,\ref{24},\,\ref{31},\,\ref{32},\,\ref{33},\,\ref{32*},\,\ref{33*}), we could deduce that for $\infty>p>r>1$, $ 1<t\leq \delta^{-\varepsilon}$
\begin{eqnarray*}
\big\|M^t_{\delta S_{\alpha,\beta}} f \big\|_p&\lesssim&_{p,n,\varphi,\varepsilon}\sum_{k=2}^{\infty}\delta^{-(n/r)\varepsilon}\delta^{k\varepsilon}\left|\int_{\RR^n}|f\ast\varphi_{\delta^{(k+3)\varepsilon}\delta t}|^p(x)dx\right|^{1/p}
\\&+& \sum_{k=0}^{s}\delta^{-(n/r)\varepsilon}\delta^{-4(\frac{n}{r}+1)\varepsilon}2^{-k}\left|\int_{\RR^n}|f\ast\varphi_{2^{-(k+1)}\delta t}|^p(x)dx\right|^{1/p}
\\&\lesssim&_{p,n,\varphi,\varepsilon}\left(\frac{1}{\delta}\right)^{5\left(\frac{n}{r}+2\right)\varepsilon}\big\|f\big\|_p.
\end{eqnarray*}
This proves the Theorem.
$\hfill\blacksquare$
\end{proof}

\begin{theorem}\label{50}
For $\infty>p>1$, $ 0<t\leq \delta^{-\varepsilon}$, $f(x)\in L^p(\RR^n)$  and $supp\,\hat{f}(\xi)\subseteq \left(B(0,t^{-1}\delta^{-(1+4\varepsilon)})\right)^c$. Then with appropriate $\alpha$, $\beta$ depending on $\varepsilon,\ n,\ p$,  we could obtain:
$$\big\|M^t_{\delta S_{\alpha,\beta}} f \big\|_p\lesssim_{p,n,\varphi,\varepsilon}\big\|f\big\|_p.$$

\end{theorem}

\begin{proof} By Formula\,(\ref{26}), we could write $M^t_{\delta S_{\alpha,\beta}} f$ as following:
\begin{eqnarray}\label{51}
\left|M^t_{\delta S_{\alpha,\beta}} f(x)\right| &\leq&\sum_{k=2}^{\infty}\sup_{ A\in O(\RR^n)}\left|\int_{\RR^n}t^{-n} \eta_1^k((uA^{-1})/t) f\ast\varphi_{\delta^{(k+3)\varepsilon}\delta t}(x-u)du\right|
\\  \nonumber&+&\sum_{k=0}^{s}\sup_{ A\in O(\RR^n)}\left|\int_{\RR^n}t^{-n} \eta_0^k((uA^{-1})/t) f\ast\varphi_{2^{-(k+1)}\delta t}(x-u)du\right|.
\end{eqnarray}
Notice that $supp\,\hat{f}(\xi)\subseteq \left(B(0,t^{-1}\delta^{-(1+4\varepsilon)})\right)^c$ and  $supp\,\mathfrak{F}\left(\varphi_{2^{-(k+1)}\delta t}\right)(\xi)\subseteq B(0, t^{-1}\delta^{-(1+4\varepsilon)})$ for $0\leq k\leq s,\,k\in\ZZ$, thus we could deduce that $f\ast\varphi_{2^{-(k+1)}\delta t}=0$. From Formula\,(\ref{51}), we could obtain
\begin{eqnarray}\label{52}
\left|M^t_{\delta S_{\alpha,\beta}} f(x)\right| &\leq&\sum_{k=2}^{\infty}\sup_{ A\in O(\RR^n)}\left|\int_{\RR^n}t^{-n} \eta_1^k((uA^{-1})/t) f\ast\varphi_{\delta^{(k+3)\varepsilon}\delta t}(x-u)du\right|.
\end{eqnarray}
By Formulas\,(\ref{22}) and Lemma\,\ref{oooh1}, we could have
\begin{eqnarray}\label{53}
& &\left[\int_{\RR^n}\left(\sum_{k=2}^{\infty}\sup_{ A\in O(\RR^n)}\left|\int_{\RR^n}t^{-n} \eta_1^k((uA^{-1})/t) f\ast\varphi_{\delta^{(k+3)\varepsilon}\delta t}(x-u)du\right|\right)^pdx\right]^{1/p}
\\ \nonumber &\leq& \left[\int_{\RR^n}\left(\sum_{k=2}^{\infty}\sup_{\mathfrak{t}>0, A\in O(\RR^n)}\left|\int_{\RR^n}\mathfrak{t}^{-n} \eta_1^k((uA^{-1})/\mathfrak{t}) f\ast\varphi_{\delta^{(k+3)\varepsilon}\delta \mathfrak{t}}(x-u)du\right|\right)^pdx\right]^{1/p}
\\  \nonumber&\leq&\left[\int_{\RR^n}\left(\sum_{k=2}^{\infty}\sup_{\mathfrak{t}>0, A\in O(\RR^n)}\left|M_{\varphi N}^{\ast\ast}f(x)\int_{\RR^n}\mathfrak{t}^{-n} \left|\eta_1^k((uA^{-1})/\mathfrak{t})\right|\left(1+\frac{|u|_{e}}{\delta^{(k+3)\varepsilon}\delta \mathfrak{t}}\right)^{N}du\right|\right)^pdx\right]^{1/p}
\\  \nonumber&\lesssim&_{p,N,n,\varphi,\varepsilon}\|f\|_p\ \ \hbox{for}\  p>1, N> n/p, N\in\RR.
\end{eqnarray}

Let $N$ be $N=\frac{n}{p}+\varepsilon$, from Formulas\,(\ref{51},\,\ref{52},\,\ref{53}), we could prove the Theorem\,\ref{50}.
$\hfill\blacksquare$
\end{proof}

\end{document}